\documentclass[oneside,10pt]{amsart}
\usepackage{preamble}

\title{Universal minimal flows of homeomorphism groups of continua}
\author{Sumun Iyer}
\date{\today}
\email{sumuni@andrew.cmu.edu}

\begin{document}

\begin{abstract}
    We define a combinatorial property of a projective Fraisse category which we call the \emph{approximate Ramsey property}. Let $F$ be a continuum, $G$ a closed subgroup of the homeomorphism group of $F$, and  $\mathbb{F}$ the limit of projective Fraisse category $\mathcal{F}$ such that $\Aut(\limF)$ is dense in $G$. We prove that $\mathcal{F}$ has the approximate Ramsey property if and only if $G$ is extremely amenable. We prove that the group of homeomorphisms of the universal pseudo-solenoid has non-metrizable universal minimal flow. 
\end{abstract}
\maketitle

\section{Introduction}

All spaces are assumed Hausdorff. Let $G$ be a topological group. A \textbf{$G$-flow} is a compact space $X$ equipped with a continuous $G$-action. The $G$-flow $X$ is \textbf{minimal} when every orbit in it is dense. By a theorem of Ellis, for each topological group $G$ there is a unique up to isomorphism minimal $G$-flow $\mathcal{M}(G)$ with the property that for any minimal $G$-flow $X$ there is a continuous equivariant surjection $\mathcal{M}(G) \to X$ (see \cite[Theorem 6.1.3]{pestov_book}). This flow $\mathcal{M}(G)$ is called the \textbf{universal minimal flow of $G$}. $\mathcal{M}(G)$ is in many cases very large, with wild dynamics; for any locally compact, non-compact group $G$, for example, $\mathcal{M}(G)$ is a non-metrizable space by a theorem of Veech (see \cite{kpt} Appendix). However for many natural non-locally compact groups the universal minimal flow can be a  well-understood action of the group on a metrizable space or even the trivial action on the singleton space. In this latter case, when $\norm{\mathcal{M}(G)}=1$, the group $G$ is called \textbf{extremely amenable}. The reader can find in \cite{pestov_book} much more on the topic, many examples, and the connections between extreme amenability, combinatorics, and probability.

A \textbf{continuum} is a compact, connected, metrizable space. Given a continuum $F$, let $\Homeo(F)$ be the group of all homeomorphisms of $F$ with the uniform convergence topology. We are interested here in the universal minimal flows of the groups $\Homeo(F)$, for $F$ a continuum. A topological group is \textbf{non-archimedean} if it has a neighborhood basis at the identity consisting of open subgroups. In \cite{kpt}, Kechris-Pestov-Todorcevic characterized extreme amenability of non-archimedean Polish groups by a combinatorial condition called the Ramsey property. We remark that $\Homeo(F)$ is typically \emph{not} a non-archimedean group. 

\textbf{Projective Fraisse categories} are a tool developed by Irwin-Solecki \cite{irwin_solecki} to study continua and their homeomorphism groups through categories of finite graphs. We give the reader a rough idea of how projective Fraisse categories work here, but the precise definitions are in Section \ref{sec:proj_fr_lim}. The general set-up is that we have a category $\mathcal{F}$ whose objects are finite graphs (possibly enriched with extra combinatorial information) and morphisms which are graph morphisms. From category $\mathcal{F}$ is produced a limit object $\mathbb{F}$ called the \textbf{projective Fraisse limit of $\mathcal{F}$}. This limit $\mathbb{F}$ is topologically a Cantor set and it comes with an compact binary relation $R^\mathbb{F}$ which it inherits from the edges of the graphs in $\mathcal{F}$. In the situations of interest, the relation $R^\mathbb{F}$ is actually an equivalence relation and $\mathbb{F}$ quotiented by $R^\mathbb{F}$ is some continuum $F$ that we wish to understand:
\[q: \mathbb{F} \to \mathbb{F}/ R^{\mathbb{F}}=F\]
The quotient map $q$ induces a natural map:
\[\Phi:\Aut(\mathbb{F}) \to \Homeo(F)\]
which goes from the group of automorphisms of structure $\mathbb{F}$ to the group we care about- $\Homeo(F)$. In situations of interest, the image of $\Phi$ is \emph{dense} in $\Homeo(F)$ or in some natural subgroup of $\Homeo(F)$. As presented above, the theory might seem ad-hoc but in fact in \cite{panagio_cpt} it is shown that for any compact metric space $F$ and any closed subgroup $G$ of $\Homeo(F)$, there is some projective Fraisse category whose limit quotients onto $F$ and such that $\Phi[\Aut(\limF)]$ is dense in $G$. 

The main theorem isolates a combinatorial condition of the category $\mathcal{F}$ which is equivalent to extreme amenability of closed subgroup $\overline{\Phi[\Aut(\mathbb{F})]} \leq \Homeo(F)$. So for instance if the image of $\Phi$ is dense in $\Homeo(F)$, the theorem below characterizes extreme amenability of the group $\Homeo(F)$.  

\begin{thm}\label{thm_intro}
    Assuming the set-up above, the category $\mathcal{F}$ has the approximate Ramsey property if and only if $\overline{\Phi[\Aut(\limF)]}$ is extremely amenable.
\end{thm}

For the definition of the \textbf{approximate Ramsey property}, see Definition \ref{defn_arp}. Here we remark that the approximate Ramsey property is a completely combinatorial property about category $\mathcal{F}$ and that it is a weakening of the usually considered \textbf{Ramsey property} for projective Fraisse categories as defined in, for example \cite{bk_gower}. From the combinatorial perspective, the approximate Ramsey property might be of interest as a new and naturally motivated Ramsey-type notion. Theorem \ref{thm_intro} is proved in Section \ref{sec_main_thms}, as Corollary \ref{cor_main_ea}. The format of Theorem \ref{thm_intro}- characterizing a property of $\Homeo(F)$ via combinatorial conditions on $\mathcal{F}$- is similar to a recent work of Poór-Solecki \cite{poor_solecki} which gives, among other results, a combinatorial condition equivalent to the existence of a comeager conjugacy class in $\Homeo(F)$. Both Theorem \ref{thm_intro} and the results of \cite{poor_solecki} exploit the combinatorial nature of the group $\Aut(\limF)$ along with control of the topology $\Aut(\limF)$ inherits \emph{from the group $\Homeo(F)$ via map $\Phi$}.  

In Section \ref{sec_main_thms} we also prove a generalization of Theorem \ref{thm_intro} which characterizes metrizability of the universal minimal flow-- Corollary \ref{cor_main_metr_umf}. The generalization is natural from the perspective of Ramsey theory and the results of \cite{zucker_umf}. 

Finally as an application of the tools presented here we prove the following:

\begin{thm}\label{thm_intro_2}
    Let $S$ be the universal pseudo-solenoid. The group $\Homeo(S)$ of homeomorphisms of $S$ with the uniform convergence topology has non-metrizable universal minimal flow.
\end{thm}

Theorem \ref{thm_intro_2} is proved in Section \ref{sec:applications} where it occurs as Theorem \ref{thm_pseudo_sol}. A \textbf{pseudo-solenoid}, also sometimes called \textbf{pseudo-circle} in the literature, is a certain type of continuum, see Definition \ref{defn_solenoid} for the precise definition. These spaces are ``circle-shaped'' in that every open cover of a pseudo-solenoid can be refined by a cover of the form $U_0,U_1,\ldots, U_{n-1}$ where $U_i \cap U_j \neq \emptyset$ if and only if $\norm{i-j} \leq 1$ or $\{i,j\}=\{0,n-1\}$. A continuum $F$ is \textbf{indecomposable} if for any $A,B$ subcontinua of $F$ with $F=A \cup B$, we have that $A=F$ or $B=F$. Pseudo-solenoids are \textbf{hereditarily indecomposable} which means that every subcontinuum of a pseudo-solenoid is indecomposable. This is a very strong and very strange property for a continuum to have. Note that the interval $[0,1]$ or any manifold is decomposable. Being hereditarily indecomposable indicates that pseudo-solenoids are topologically very complex. The \textbf{universal} pseudo-solenoid is a pseudo-solenoid which continuously surjects onto any pseudo-solenoid. A projective Fraisse category approximating the universal pseudo-solenoid was constructed by Irwin in \cite{irwin_thesis} and we use Irwin's category to prove Theorem \ref{thm_intro_2}.

\subsection{Notation} We use $A \fsubset X$ to mean $A$ is a finite subset of $X$. Throughout, $\N=\{0,1,2,3,\ldots\}$. 

\section{Projective Fraisse limits}\label{sec:proj_fr_lim}
In this section, we introduce projective Fraisse limits and the general setting in which the results of Section \ref{sec_main_thms} apply.

\subsection{Background on projective Fraisse limits and their topological realizations} 
We follow for the most part in this subsection the treatment of projective Fraisse categories from \cite{pan_sol}.

Let $A$ be a set and $R^A$ a binary relation on $A$. We will say that $(A,R^A)$ is a \textbf{graph} if $R^A$ is reflexsive and symmetric. As usual, \textbf{reflexsive} means for all $a \in A$, $aR^Aa$ and symmetric means that for all $a,b \in A$, $aR^Ab \implies bR^A a$.  Note that this convention is different from the usual combinatorial graph theory definition, as we require our graphs to have self-loops at each vertex.

Let $\mathcal{L}$ be a language containing a special binary relation symbol $R$. If $A$ and $B$ are $\mathcal{L}$-structures, $f:B \to A$ is a \textbf{morphism} means that $f$ preserves the symbols in $\mathcal{L}$. That is, if $S \in \mathcal{L}$ is a relation symbol of arity $m$, then for all $b_1,\ldots, b_m \in B$, $S^B(b_1,\ldots, b_m) \implies S^A(f(b_1),\ldots, f(b_m))$. If $\sigma \in \mathcal{L}$ is a function symbol of arity $m$, then for all $b_1,\ldots,b_m \in B$ we have $f(\sigma^B(b_1,\ldots,b_m))=\sigma^A (f(b_1),\ldots,f(b_m))$. We say $f$ is a \textbf{surjective morphism} if $f:B \to A$ is surjective and for any relation symbol $S\in \mathcal{L}$ of arity $m$, if $S^A(a_1,\ldots, a_m)$ then there exists $b_1,\ldots, b_m$ with $f(b_i)=a_i$ and $S^B(b_1,\ldots,b_m)$.

A \textbf{projective Fraisse category} is a category of $\mathcal{L}$-structures and surjective morphisms such that:
\begin{enumerate}[(i)]
    \item for each $A \in \mathcal{F}$, $(A,R^A)$ is a finite graph
    \item there are countably many structures in $\mathcal{F}$ up to isomorphism
    \item for any $A,B \in \mathcal{F}$ there exists $C \in \mathcal{F}$ and morphisms $C \to A$ and $C \to B$ in $\mathcal{F}$
    \item for any $A,B,C \in \mathcal{F}$ and morphisms $f:B \to A$ and $g:C \to A$ in $\mathcal{F}$ there is a structure $D$ and morphisms $f': D \to B$ and $g': D \to C$ in $\mathcal{F}$ such that $f \circ f'=g \circ g'$. 
\end{enumerate}

The last property above is called the \textbf{projective amalgamation property} and it is the most important one in the definition. We use $\Epi(B,A)$ to represent the set of all morphisms $B \to A$ in $\mathcal{F}$.

First some notation for inverse limits. All inverse limits in the paper are over sequences of spaces indexed by the natural numbers. For any inverse limit of the form $L=\varprojlim (A_n,f_n^{n+1})$ where $f_n^{n+1}:A_{n+1} \to A_n$ we use $\pi_n^L:L \to A_n$ to denote projection onto the $n$th coordinate. Given $L=\varprojlim (A_n,f_n^{n+1})$ we use the notation that for any $n>m$
\[f_m^n: A_n \to A_m \textrm{ is the map } f_m^n = f_m^{m+1} \circ f_{m+1}^{m+2} \circ \cdots \circ f_{n-1}^n\]
along with the convention that for $n=m$, $f_m^n =\textrm{id}_{A_m}$.

Now we extend the projective Fraisse category to a larger category which should be thought off as potential limit objects for $\mathcal{F}$.

Given $\mathcal{F}$ a projective Fraisse category, define category $\mathcal{F}^\omega$ as follows. The category $\mathcal{F}^\omega$ consists of all $\mathbb{L}$ such that $\mathbb{L}=\varprojlim (A_n,f_n^{n+1})$ where each $A_n$ is a structure in $\mathcal{F}$ and each $f_n^{n+1}:A_{n+1} \to A_n$ is a morphism in $\mathcal{F}$. So the elements of $\mathcal{F} ^\omega$ are inverse limits of inverse sequences in the category $\mathcal{F}$. Morphisms in category $\mathcal{F}^\omega$ are defined like this. For $\mathbb{L}=\varprojlim (A_n,f_n^{n+1})$ and $\mathbb{M} = \varprojlim (B_n,g_n^{n+1})$, then $h:\mathbb{L} \to \mathbb{M}$ is a \textbf{morphism} if there is an increasing sequence $i_0<i_1<i_2<\cdots$ and morphisms $h_n: A_{i_n} \to B_n$ in $\mathcal{F}$ such that for any $n$, $\pi^{\mathbb{M}}_n \circ h = h_n \circ \pi^{\mathbb{L}}_{i_n}$. A morphism $h:\mathbb{L} \to \mathbb{M}$ is an \textbf{isomorphism} if $h$ is a morphism which is additionally a bijection and $h^{-1}$ is a morphism $\mathbb{M} \to \mathbb{L}$. We note the following fact which is well-known in projective Fraisse theory and easy to derive directly from the definition of isomorphism, see \cite{pan_sol}. 

\begin{prop}\label{prop_auts_are_diag}
    A map $h: \mathbb{L}=\varprojlim (A_n,f_n^{n+1}) \to \mathbb{M} =\varprojlim (B_n,g_n^{n+1})$ is an isomorphism iff there exist natural numbers $i_0<i_1<i_2<i_3<
    \cdots$ and morphisms $\{h_j\}_{j\in\N}$ in $\mathcal{F}$ such that $h_j: A_{i_{j+1}} \to B_{i_j}$ for $j$ even and $h_j:B_{i_{j+1}} \to A_{i_j}$ for $j$ odd; for every even $j$, $h_j \circ h_{j+1}=g_{i_j}^{i_{j+2}}$ and for every odd $j$, $h_j \circ h_{j+1}=f_{i_j}^{i_{j+2}}$; and for every even $j$ $\pi_{i_j}^{\mathbb{M}} \circ h = h_j \circ \pi_{i_{j+1}}^\mathbb{L}$. 
\end{prop}

Note $\mathcal{F}^\omega$ contains the structures in $\mathcal{F}$ (just take inverse limits with the identity as each bonding map) and for any morphism $f:A \to B$ in $\mathcal{F}^\omega$ with $A,B$ in $\mathcal{F}$ we have that $f$ is in $\mathcal{F}$. Further, a morphism $f: \mathbb{L} \to A$ in $\mathcal{F}^\omega$ where $A \in \mathcal{F}$ is exactly a map of the form $f= f_n \circ \pi^{\mathbb{L}}_n$ for some $f_n \in \mathcal{F}$.

Now we add topological information to $\mathcal{F}^\omega$. We take all finite structures $A$ in $\mathcal{F}$ with the discrete topology. Then, any $\mathbb{L}=\varprojlim(A_n,f_n^{n+1})$ inherits its topology from the product topology on $\prod_{n \in \N} A_n$ and in particular, is a compact, metrizable, zero-dimensional space. Note that morphisms in $\mathcal{F}^\omega$ are continuous and isomorphisms are homeomorphisms because a continuous bijection between compact Hausdorff spaces is automatically a homeomorphism.

Finally, note that any $\mathbb{L}=\varprojlim(A_n,f_n^{n+1})$ is an $\mathcal{L}$-structure in the following sense. For any relation symbol $S \in \mathcal{L}$ define:
\[(x_n)_{n\in\N}S^{\mathbb{L}}(y_n)_{n\in\N} \iff \forall n \in \N \ (x_n S^{A_n} y_n)\]

and for any function symbol $\sigma \in \mathcal{L}$ of arity $m$ define:
\[\sigma^{\mathbb{L}}((x^1_n)_{n\in \N}, (x^2_n)_{n\in\N}, \ldots, (x^m_n)_{n\in\N}) = (\sigma^{A_n}(x_n^1, \ldots , x_n^m))_{n\in\N} \]

In particular, $R^{\mathbb{L}}$ is a compact, reflexsive, symmetric binary relation on $\mathbb{L}$. Further, morphisms in $\mathcal{F}^\omega$ are model-theoretic morphisms in that they preserve the relations and functions in $\mathcal{L}$.

A main theorem of \cite{irwin_solecki} is the following:

\begin{thm}[Irwin-Solecki \cite{irwin_solecki}, see also Panagiotoupoulos-Solecki \cite{pan_sol} ]\label{thm_irw_sol}
    Let $\mathcal{F}$ be a projective Fraisse category. There is a unique up to isomorphism $\mathbb{F} \in \mathcal{F}^\omega$ such that:
    \begin{enumerate}[(i)]
        \item for each $A \in \mathcal{F}$, there is a morphism in $\mathcal{F}^\omega$ from $\mathbb{F} \to A$
        \item (projective extension property) for each $A, B \in \mathcal{F}$ and morphisms $f: \mathbb{F} \to A$ and $g:B \to A$, there is a morphism $h:\mathbb{F} \to B$ with $g \circ h=f$
    \end{enumerate}
\end{thm}

The unique up to isomorphism structure $\mathbb{F}$ as in Theorem \ref{thm_irw_sol} is called the \textbf{projective Fraisse limit of $\mathcal{F}$}. The following application of the projective extension property is well-known in projective Fraisse theory:

\begin{lem}\label{lem_ext_to_aut}
    Let $\mathbb{F}=\varprojlim (A_n,\varphi_n^{n+1})$ be the projective Fraisse limit of category $\mathcal{F}$ with projection maps $\pi_n:\mathbb{F} \to A_n$. For any $n\geq m$, $f:A_n \to A_m$ a morphism in $\mathcal{F}$, there exists $\tilde{f}$ an automorphism of $\mathbb{F}$ such that $\pi_m \circ \tilde{f}=f\circ \pi_n$.
\end{lem}


\subsection{Approximate Ramsey property}

For any $A$ in a projective Fraisse category $\mathcal{F}$ we define a metric on $A$, denoted $d_A$ using the graph relation $R^A$. Precisely, for any $a \in A$ $d_A(a,a)=0$ and for any $a \neq b$ in $A$, $d_A(a,b)=l$ if $l$ is the least natural number such that there exists $a_0,a_1,\ldots, a_l \in A$ with $a_0=a$, $a_l=b$ and $a_iR^A a_{i+1}$ for all $i\leq l$. 

For $A,B \in \mathcal{F}$ we set the metric $d_{\Epi(B,A)}$ on the set $\Epi(B,A)$ by:
\[d_{\Epi(B,A)}(f,g) = \sup_{b\in B} d_A(f(b),g(b))\]

In a metric space $(X,d)$, for any $\varepsilon>0$ and $A \subseteq X$ we use the notation 
\[[A]_\varepsilon= \{x \in X \ : \ \exists a \in A \ (d(a,x)\leq \varepsilon\}\]

Here is the main new combinatorial definition:

\begin{defn}\label{defn_arp}
    A projective Fraisse category $\mathcal{F}$ has the \textbf{approximate Ramsey property} if for any $A, B \in \mathcal{F}$ and any $d \in \N$ there exists $C \in \mathcal{F}$ such that for any $\chi:\Epi(C,A) \to \{1,2,\ldots, d\}$ there exists $f \in \Epi(C,B)$ and $1\leq i \leq d$ such that 
    \[\Epi(B,A) \circ f \subseteq [\chi^{-1}(i)]_1\]
\end{defn}

We say $C \in \mathcal{F}$ \textbf{witnesses the approximate Ramsey property for $A,B,d$} if for any coloring $\chi:\Epi(C,A) \to \{1,\ldots, d\}$ there is $f \in \Epi(C,B)$ and $i \in \{1,\ldots, d\}$ such that $\Epi(B,A) \circ f \subset [\chi^{-1}(i)]_1$.

Note that the approximate Ramsey property is a weakening of the usual Ramsey property considered for projective Fraisse categories.

Finally we have the following weakening of the approximate Ramsey property which is very natural from the perspective of Ramsey theory:

\begin{defn}
    Let $A$ in $\mathcal{F}$. Then $A$ is said to have \textbf{approximate Ramsey degree $\leq k$} if for any $B \in \mathcal{F}$ and $d \in \N$ there exists $C \in \mathcal{F}$ such that for any $\chi:\Epi(C,A) \to \{1,2,\ldots, d\}$ there is $f \in \Epi(C,B)$ and $I \subseteq \{1,2,\ldots, d\}$ with $\norm{I} \leq k$ such that 
    \[\Epi(B,A) \circ f \subseteq [\chi^{-1}[I]]_1\]
\end{defn}

We define \textbf{the approximate Ramsey degree of $A$} to be the minimum $k \in \N$ such that $A$ has approximate Ramsey degree $\leq k$ and write $\textrm{Ardeg}(A) =k$. We say $C$ \textbf{witnesses that $\rdeg(A) \leq k$ for} $B,d$ if for any coloring $\chi:\Epi(C,A) \to \{1,2,\ldots,d\}$, there is $f \in \Epi(C,B)$ and $I \subseteq \{1,2,\ldots,d\}$ with $\norm{I} \leq k$ such that $\Epi(B,A) \circ f \subseteq [\chi^{-1}(I)]_1$.

If no such $k$ exists (i.e., the approximate Ramsey degree of $A$ is not $\leq k$ for all $k \in \N$) then we say $A$ has \textbf{infinite approximate Ramsey degree} and write $\Ard(A) =\infty$. 

\begin{rmk}\label{rmk_arp_extend}
    If $C \in \mathcal{F}$ witnesses the approximate Ramsey property for $A,B,d$ and $D \in \mathcal{F}$ is such that $\Epi(D,C) \neq \emptyset$, then $D$ witnesses the approximate Ramsey property for $A,B,d$. Similarly if $C$ witnesses that $\rdeg(A) \leq k$ for $B,d$, and $D$ is such that $\Epi(D,C) \neq \emptyset$, then $D$ witnesses that $\rdeg(A) \leq k$ for $B,d$. 
\end{rmk}

\subsection{Transitive projective Fraisse limits}\label{subsec_ntn_limit}

\begin{defn}
    A projective Fraisse category $\mathcal{F}$ is \textbf{transitive} if $R^\mathbb{F}$ is a transitive relation, where $\mathbb{F}$ is the projective Fraisse limit of $\mathcal{F}$. 
\end{defn}

Note that for a transitive projective Fraisse category with limit $\mathbb{F}$, the relation $R^\mathbb{F}$ is a compact equivalence relation. Thus, the quotient $\mathbb{F} / R^\mathbb{F}$ equipped with the quotient topology is a compact, metrizable space. We denote $\mathbb{F} / R^\mathbb{F}$ by $F$ and we use $q:\mathbb{F} \to F$ to denote the quotient map. We fix a metric $d_F$ on $F$. 

We will now investigate in some detail how the structure of $\mathbb{F}$ and the topology on $F$ interact. Note that since $\mathbb{F}$ is in $\mathcal{F}^\omega$, we have that
\[\mathbb{F}= \varprojlim (A_n, \varphi_n^{n+1})\]
where each $A_n$ is a structure in $\mathcal{F}$ and $\varphi_n^{n+1}:A_{n+1} \to A_n$ is a morphism in $\mathcal{F}$.
We use $\pi_n:\mathbb{F} \to A_n$ as the $n$th projection map. 

\begin{lem}\label{lem_rel_disj}
    Let $a,b \in A_n$. We have:
    \[q[\pi_n^{-1}(a)] \cap q[\pi_n^{-1}(b)] \neq \emptyset \iff aR^{A_n}b\]
\end{lem}

\begin{proof}
   Suppose that $q[\pi_n^{-1}(a)] \cap q[\pi_n^{-1}(b)] \neq \emptyset$. Let $x \in q[\pi_n^{-1}(a)] \cap q[\pi_n^{-1}(b)]$. Then, $x= q((x_n)_{n \in\N})$ with $x_n=a$ and $x=q(y_n)_{n\in\N})$ with $y_y =b$. So $(x_n)_{n \in\N} R^{\mathbb{F}} (y_n)_{n\in\N}$ which implies $x_n R^{A_n}y_n$ i.e, $a R^{A_n}b$. 

   For the converse direction, suppose $a R^{A_n} b$. Let $\mathcal{T}$ be the set of all finite strings of pairs:
   \[\langle (x_n,y_n),(x_{n+1},y_{n+1}),\ldots, (x_N,y_N) \rangle\]
   where for all $n \leq m\leq N$
   \begin{enumerate}
       \item $x_n=a$ and $y_n =b$
       \item $x_m,y_m \in A_m$ with $x_m \neq y_m$ for all $m$
       \item $x_mR^{A_m}y_m$
       \item $\varphi_m(x_{m+1})=x_m$ and $\varphi_m(y_{m+1})=y_m$
   \end{enumerate}

   Under the extension relation, $\mathcal{T}$ is a tree as it is easy to check that an initial substring of an element of $\mathcal{T}$ is an element of $\mathcal{T}$. Since each $A_m$ is finite, $\mathcal{T}$ is finitely splitting, i.e., for each 
   \[\langle (x_n,y_n),(x_{n+1},y_{n+1}),\ldots, (x_N,y_N) \rangle \in \mathcal{T}\]
   there are only finitely many pairs $(x_{N+1},y_{N+1})$ such that 
   \[\langle (x_n,y_n),(x_{n+1},y_{n+1}),\ldots, (x_N,y_N), (x_{N+1},y_{N+1}) \rangle \in \mathcal{T}\] 
   The tree $\mathcal{T}$ is infinite because given any element  $\langle (x_n,y_n),(x_{n+1},y_{n+1}),\ldots, (x_N,y_N) \rangle$, because we know that $\varphi_N$ is a surjective morphism and so surjective on relation $R^{A_N}$, there is some $z,w \in A_{N+1}$ with $zR^{A_{N+1}}w$ and $\varphi_N(z)=x_N$ and $\varphi_N(w)=y_N$. Because $x_N \neq y_N$ by condition (2) of elements of $\mathcal{T}$ we know $z \neq w$. So  $\langle (x_n,y_n),(x_{n+1},y_{n+1}),\ldots, (x_N,y_N), (z,w) \rangle \in \mathcal{T}$. By Konig's lemma, $\mathcal{T}$ has an infinite branch
   \[\langle (x_n,y_n),(x_{n+1},y_{n+1}),\ldots \rangle\]
   Let $\mathbf{x}, \mathbf{y \in \mathbb{F}}$ with $\pi_m(\mathbf{x})=x_m$ and $\pi_m(\mathbf{y})=y_m$ for all $m \geq n$. Then, $\pi_n(\mathbf{x})=a$ and $\pi_n(\mathbf{y})=b$ by condition (1) on $\mathcal{T}$ and $q(\mathbf{x})=q(\mathbf{y})$ by condition (3) on $\mathcal{T}$. 
\end{proof}

For a metric space $(X,d)$ and $A,B \subseteq X$, let $d(A,B)=\inf \{d(a,b) \ : a\in A, b\in B\}$. Let $\textrm{diam}_d(A)=\sup\{d(a,a') \ : \ a,a' \in A\}$.

\begin{lem}\label{lem_metric_spread}
    For any $n \in \N$, there exists $\varepsilon >0$ such that for all $\mbf{x},\mbf{y}\in \mathbb{F}$
    \[d_F(q(\mbf{x}),q(\mbf{y}))<\varepsilon \implies \pi_n(\mbf{x}) R^{A_n} \pi_n(\mbf{y})\]
\end{lem}

\begin{proof}
    Let $n \in \N$. For any $a,b \in A_n$ with $\neg(a R^{A_n}b)$, we have by Lemma \ref{lem_rel_disj} that $q[\pi_n^{-1}(a)]$ and $q[\pi_n^{-1}(b)]$ are disjoint. Note $q[\pi_n^{-1}(a)]$ and $q[\pi_n^{-1}(b)]$ are compact as well. So $d_F(q[\pi_n^{-1}(a)], q[\pi_n^{-1}(b)])>0$. Since $A_n$ is finite let $\varepsilon$ such that
    \[0<\varepsilon <\min \{d_F(q[\pi_n^{-1}(a)], q[\pi_n^{-1}(b)]) \ : \ a,b \in A_n, \ \neg( a  R^{A_n}b) \}\]
    This $\varepsilon$ works to witness the lemma.
\end{proof}

For the next lemma note the following fact. If $K_n$ is a sequence of compact sets in a compact metric space $(X,d)$ with $K_{n+1} \subseteq K_n$ and $\bigcap_{n\in\N}K_n$ is a singleton, then $\textrm{diam}_d(K_n) \to 0$ as $n \to \infty$.

\begin{lem}\label{lem_graphdistsmall}
    Fix $M \in \N$. For all $\varepsilon>0$, there exists $N \in \N$ such that for all $\mbf{x},\mbf{y} \in \mathbb{F}$:
    \[d_{A_N}(\pi_N(\mbf{x}),\pi_N(\mbf{y}))\leq M \implies d_F(q(\mbf{x}),q(\mbf{y}))<\varepsilon\]
\end{lem}

\begin{proof}
    Fix $M,\varepsilon$. For any $\mbf{x}=(x_n)_{n\in\N} \in \mathbb{F}$, notice that 
    \[q(\mbf{x})= \bigcap_{n\in\N} q[\pi_n^{-1}(x_n)]\]
    and the sets in the intersection are compact and nested. So, $\textrm{diam}(q[\pi_n^{-1}(x_n)]) \to 0$ and there exists some $n \in \N$ so that for $m \geq n$, $\textrm{diam}(q[\pi_m^{-1}(x_m)])<\frac{\varepsilon}{M}$. Define $U_\mbf{x} = \pi_n^{-1}(x_n)$. Now $\{U_\mbf{x} \ : \ \mbf{x} \in\mathbb{F}\}$ is an open cover of $\mathbb{F}$ so by compactness there is a finite sub-cover of $\mathbb{F}$ of the form:
    \[\pi_{n_1}^{-1}(a_1), \ldots, \pi_{n_k}^{-1}(a_k)\]
    where each $\textrm{diam}(q[\pi_{n_i}^{-1}(a_i)])<\frac{\varepsilon}{M}$. Let $N = \max \{n_1,\ldots, n_k\}$. 
    
    We claim for $b \in A_N$, $\textrm{diam}(q[\pi_N^{-1}(b)])<\frac{\varepsilon}{M}$. To see this, let $b \in A_N$. Take $\mbf{b} \in \mathbb{F}$ with $\pi_N(\mbf{b})=b$. Because the displayed sets above cover, there is $i \leq k$ with $\mbf{b} \in \pi_{n_i}^{-1}(a_i)$. That is, $\pi_{n_i}(\mbf{b})=a_i$. Since $\mbf{b} \in \mathbb{F}$, we have that 
    \[\varphi_{n_i}^N(b)=a_i \implies \pi_N^{-1}(b) \subseteq \pi_{n_i}^{-1}(a_i) \]
    and so: 
    \[q[\pi_N^{-1}(b)] \subseteq q[\pi_{n_i}^{-1}(a_i)]\]
    which implies that $\textrm{diam}(q[\pi_N^{-1}(b)])<\frac{\varepsilon}{M}$. 

    Now the lemma follows by the claim we just proved and Lemma \ref{lem_rel_disj}.
\end{proof}

Finally we make a definition which is a bit technical but necessary for later sections. It assumes a fixed set up of $\mathbb{F} =\varprojlim(A_n,\varphi_n^{n+1})$. Recall the metric $d_{\Epi(B,A)}$ defined on the set $\Epi(B,A)$ of morphisms in $\mathcal{F}$ from $B$ to $A$. We define a metric $d_{\Epi(\mathbb{F},A)}$ of the set $\Epi(\mathbb{F},A)$ for any $A \in \mathcal{F}$ as follows. For $f,g \in \Epi(\mathbb{F},A)$ with $f = f' \circ \pi_n$ and $g=g' \circ \pi_m$ with $f',g' \in \mathcal{F}$, and $n\geq m$, we set:
\[d_{\Epi(\mathbb{F},A)}(f,g) = d_{\Epi(A_n,A)} (f',g' \circ \varphi_m^n) \]

The definition $d_{\Epi(\mathbb{F},A)}$ is well-defined because of the following observation.

\begin{rmk}\label{rmk_dist_cohere}
    For any $A,B,C$, $f,g \in \Epi(B,A)$ and $h \in \Epi(C,B)$, 
    \[d_{\Epi(C,A)} (f \circ h, g\circ h) = d_{\Epi(B,A)} (f,g)\]
\end{rmk}

\subsection{Groups of automorphisms and homeomorphisms}

For any compact metric space $(X,d)$ let $\Homeo(X)$ be the group of homeomorphisms of $X$ with the uniform convergence topology. This is the topology induced by the supremum metric:
\[d_{\sup}(f,g) = \sup_{x \in X} d(f(x),g(x))\]
Note that the supremum metric is \textbf{right-invariant}, i.e. for any $f,g,h$ in $\Homeo(X)$, 
\[d_{\sup}(f h,g h) = d_{\sup}(f,g)\]

\begin{rmk}\label{rmk_vs}
    Let $d$ be a right-invariant metric on a topological group $G$, and $V=B_d(1,\varepsilon)$ the open ball of radius 1 about the identity. Then, for any $S \subseteq G$,
    \[VS= \{g \in G \ : \ \exists s \in S \ d(s,g)<\varepsilon\}\]
\end{rmk}

Let $\Aut(\mathbb{F})$ be the group of automorphisms (isomorphisms) of a projective Fraisse limit $\mathbb{F}$ equipped with the topology it inherits as a subgroup of $\Homeo(\mathbb{F})$. With this topology, $\Aut(\mathbb{F})$ is a non-archimedean Polish group. Suppose that $\mathcal{F}$ is transitive. Then, the quotient map $q:\mathbb{F} \to F$ induces a map 
\[\Phi:\Aut(\mathbb{F}) \to \Homeo(F)\]
given by the formula $\Phi (f)(q(x))=q(f(x))$. This map is well-defined exactly because $f$ is an automorphism (so preserves $R^{\mathbb{F}}$) and it is immediate to check that $\Phi(f)$ is a homeomorphism (see \cite[Lemma 4.5]{irwin_solecki}). Note that the map $\Phi$ is continuous and injective but it is not a topological embedding; in particular, when considering $\Phi[\Aut(\mathbb{F})]$ as a subgroup of $\Homeo(F)$ it is equipped with a different topology than the usual non-archimedean topology on $\Aut(\limF)$.

\section{Topological dynamics}
In this section we collect the facts we need from topological dynamics.
\subsection{Dense subgroups}
Recall a group $G$ is said to be \textbf{extremely amenable} if $\mathcal{M}(G)$ is a singleton.

For $(X,\tau)$ a topological space and $Y \subseteq X$, let $(Y,\tau\restriction_Y)$ be $Y$ with the subspace topology. 

The proposition below collects two well-known facts that we will use.

    \begin{prop}\label{prop_dense_subgp}
        Let $(G,\tau)$ be a topological group and let $H$ be a subgroup of $G$ which is dense in $(G,\tau)$. Then:
        \begin{enumerate}
            \item $(G,\tau)$ is extremely amenable iff $(H,\tau\restriction_H)$ is extremely amenable
            \item $(G,\tau)$ has metrizable universal minimal flow iff $(H,\tau\restriction_H)$ has metrizable universal minimal flow
        \end{enumerate}
    \end{prop}

    \begin{proof}
        For $H$ a dense subgroup of $G$, $\mathcal{M}(H)$ and $\mathcal{M}(G)$ are isomorphic as $H$-flows and in particular are homeomorphic-- see \cite[Fact 2.3]{basso_zucker}.
    \end{proof}

\subsection{Universal minimal flows and syndetic sets}

A subset $S$ of a group $G$ is \textbf{syndetic} if there exists $F \fsubset G$ such that $F S= G$.

The following characterization of extreme amenability is due to Bartosova \cite{bartosova_thesis}:

\begin{thm}[Bartosova]\label{thm_bart_1}
    A topological group is extremely amenable iff for every open neighborhood of the identity $V \subseteq G$, and every $S,T$ open syndetic subsets of $G$, $VS \cap VT \neq\emptyset$.
\end{thm}

The following characterization of metrizability of universal minimal flows is due to \cite{dis}:

\begin{thm}[Domat-Iyer-Shinko]\label{thm_dis}
    A first-countable topological group $G$ has metrizable universal minimal flow iff for all open neighborhoods of the identity $V \subseteq G$ there exists $k \in \N$ so that for any $S_1,\ldots, S_{k+1}$ open syndetic subsets of $G$ there is $i \neq j$ such that $VS_i \cap VS_j \neq \emptyset$.
\end{thm}

\section{Combinatorics on the group}
Below we connect the combinatorics of the Fraisse category to the group $\Aut(\limF)$. The following two lemmas are technical to state but are the key for the proofs of the main theorems in Section \ref{sec_main_thms}. To parse the statements of the lemma, recall the set-up from Section \ref{subsec_ntn_limit}. Let $\mathcal{F}$ be a transitive projective Fraisse category with limit $\limF$. Let $q:\limF \to F=\limF/ R^{\limF}$ be the quotient map. Let $\Aut(\limF)$ be the group of automorphisms of $\limF$ and let $\Phi:\Aut(\limF) \to \Homeo(F)$ be the map induced by the quotient map $q$. Let $d_F$ be a metric inducing the topology on $F$ and $d_{\sup}$ the supremum metric on $\Homeo(F)$. 

In the proof of the lemmas below we use the notation that for $k \in \N$, $[k]=\{1,2,\ldots,k\}$.

\begin{lem}\label{lem_ardeg_synd_overlap}
    Suppose that $\limF=\varprojlim(A_n,\varphi_n^{n+1})$ and that $\rdeg(A_m)\leq k$ in $\mathcal{F}$. Then, for any $\varepsilon >0$ satisfying the condition below:
    \[\textrm{for all } \mbf{x},\mbf{y} \in \limF, \ d_{A_m}(\pi_m(\mbf{x}),\pi_m(\mbf{y}))\leq 2 \implies d_F(q(\mbf{x},\mbf{y}))<\varepsilon\]
    we have that for any $S_1,\ldots, S_{k+1}$ syndetic subsets of $\Aut(\limF)$, there is $i \neq j$ and $s_i \in S_i, s_j \in S_j$ such that $d_{\sup}(\Phi(s_i),\Phi(s_j))<\varepsilon$. 
\end{lem}

\begin{proof}
    Assume $A_m$ has Approximate Ramsey degree $\leq k$ in $\mathcal{F}$ and let $\varepsilon >0$ be such that for all  $\mbf{x},\mbf{y} \in \limF$,
    \begin{equation}\label{eqn_4_1}
     \ d_{A_m}(\pi_m(\mbf{x}),\pi_m(\mbf{y}))\leq 2 \implies d_F(q(\mbf{x},\mbf{y}))<\varepsilon
    \end{equation}

    Let $S_1,\dots, S_{k+1}$ be syndetic subsets of $\Aut(\limF)$. 

    Let $K\fsubset \Aut(\limF)$ such that $K^{-1}S_i=\Aut(\limF)$ for all $i=1,\ldots, k+1$. Here $K^{-1}=\{f^{-1} \ : \ f\in K\}$. By Proposition \ref{prop_auts_are_diag}, let $n>m$ such that for every $f \in K$ there exists $f_m^n \in \Epi(A_n,A_m)$ such that $\pi_m\circ f=f_m^n\circ \pi_n$. 

    Since $\rdeg(A_m)\leq k$, there exists $A_N$ with $N > n$ such that $A_N$ witnesses that $\rdeg(A_m) \leq k$ for $A_n,k+2$. That is, for any $\chi:\Epi(A_N,A_m) \to [k+2]$ there is $f\in \Epi(A_N,A_n)$ and $I \subset [k+2]$ with $\norm{I} \leq k$ such that $\Epi(A_n,A_m) \circ f \subseteq [\chi^{-1}(I)]_1$. The fact that the witness can be chosen to be some $A_N$ is by Remark \ref{rmk_arp_extend} and Condition (i) from Theorem \ref{thm_irw_sol}. 

    For $i=1,\ldots, k+1$, define $S_i^0 \subset \Epi(A_N,A_m)$ by:
    \[S_i^0=\{f \ : \ \exists s \in S_i \textrm{ s.t. }\pi_m \circ s=f \circ \pi_N\}\]
    Define $S_i^1 \subset \Epi(A_N,A_m)$ by:
    \[S_i^1=[S_i^0]_1= \{f  \ : \ \exists g \in S_i^0 \textrm{ s.t. }d_{\Epi(A_N,A_m)}(f,g)\leq 1\}\]

    We have two claims:

    \begin{claim}\label{claim_4_1}
        If $S_i^1 \cap S_j^1 \neq \emptyset$, then there exists $s_i \in S_i$ and $s_j \in S_j$ such that $d_{\sup}(\Phi(s_i),\Phi(s_j))<\varepsilon$.
    \end{claim}

    \begin{proof}[Proof of Claim \ref{claim_4_1}]
        Let $f \in S_i^1 \cap S_j^1$. Then, there is $g_i \in S_i^0$ and $g_j \in S_j^0$ with $d_{\Epi(A_N,A_m)}(f,g_i) \leq 1$ and $d_{\Epi(A_N,A_m)}(f,g_j) \leq 1$. Thus:
        \begin{equation}\label{eqn_4_4}
            d_{\Epi(A_N,A_m)}(g_i,g_j) \leq 2
        \end{equation}
        By definition of $S_i^0$, let $s_i \in S_i$ such that 
        \begin{equation}\label{eqn_4_2}
            \pi_m \circ s_i = g_i \circ \pi_N 
        \end{equation}
        and let $s_j \in S_j$ such that 
        \begin{equation}\label{eqn_4_3}
            \pi_m \circ s_j =g_j \circ \pi_N
        \end{equation}

        We will show that $d_{\sup}(\Phi(s_i),\Phi(s_j))<\varepsilon$. Let $x\in F$. Let $\mbf{x} \in \limF$ with $x=q(\mbf{x})$. We need to show: $d_F(\Phi(s_i)(x),\Phi(s_j)(x))<\varepsilon$ which by definition of the map $\Phi$ means we need to show:
        \[
        d_F(q(s_i(\mbf{x})),q(s_j(\mbf{x})))<\varepsilon
       \]
       We have:
       \begin{align*}
           d_{A_m}(\pi_m(s_i(\mbf{x})),\pi(s_j(\mbf{x}))) &= d_{A_m}(g_i(\pi_N(\mbf{x})),g_j(\pi_N(\mbf{x}))) & \textrm{ by Eqns \ref{eqn_4_2},\ref{eqn_4_3}}\\
           &\leq 2 & \textrm{by Eqn \ref{eqn_4_4}}\\
       \end{align*}
       So by Formula \ref{eqn_4_1}, $d_F(q(s_i(\mbf{x})),q(s_j(\mbf{x})))<\varepsilon$ and this completes the proof of the claim.
    \end{proof}

    \begin{claim}\label{claim_4_2}
        For any $g \in \Epi(A_N,A_n)$ and $i \in [k+1]$,
        \[\Epi(A_n,A_m) \circ g \cap S_i^0\neq \emptyset\]
    \end{claim}

    \begin{proof}[Proof of Claim \ref{claim_4_2}]
        Let $g \in \Epi(A_N,A_n)$ and $i\in [k+1]$. By Lemma \ref{lem_ext_to_aut}, let $h \in \Aut(\limF)$ such that $\pi_n\circ h=g \circ \pi_N$. Since $K^{-1}S_i=\Aut(\limF)$, there exists $f \in K$ such that $fh \in S_i$. By choice of $n$, there is $f_m^n \in \Epi(A_n,A_m)$ such that $\pi_m \circ f=f_m^n \circ \pi_n$. Now we compute that:
        \[\pi_m \circ f\circ h= f_m^n \circ \pi_n \circ h = f_m^n \circ g \circ \pi_N\]
        The fact that $fh \in S_i$ and the equation above imply that $f_m^n \circ g \in S_i^0$. So $\Epi(A_n,A_m) \circ g \cap S_i^0 \neq \emptyset$ which finishes the proof of the claim.
    \end{proof}

    We will show that there is $i \neq j$ such that $S_i^1 \cap S_j^1 \neq \emptyset$. Define $\chi:\Epi(A_N,A_m) \to [k+2]$ by:
    \[\chi(f) =\begin{cases}
        i & \textrm{ if } f \in S_i^1\\
        k+2 & \textrm{ if } f \notin \bigcup_{i=1}^{k+1}S_i^1 \\
    \end{cases}\]
    If the coloring function $\chi$ above is not well-defined, this means there is $f \in S_i^1 \cap S_j^1$ for some $i\neq j$ and we are done.  

    By the choice of $A_N$, let $g \in \Epi(A_N,A_n)$ and $I \subseteq [k+2]$ with $\norm{I} \leq k$ such that:
    
    \begin{equation}\label{eqn_5_1}
    \Epi(A_n,A_m) \circ g \subseteq [\chi^{-1}(I)]_1
    \end{equation}
    
    Since $\norm{I} \leq k$, there exists $j \in [k+1]$ such that $j \notin I$.

    By Claim \ref{claim_4_2}, let $s \in \Epi(A_n,A_m) \circ g \cap S_j^0$. By Equation \ref{eqn_5_1}, let $i \in I$ and $f \in \Epi(A_N,A_m)$ such that $\chi(f)=i$ and $d_{\Epi(A_N,A_m)}(f,s)\leq 1$. Note that $i \in I \implies i \neq j$. Since $s \in S_j^0$ and $d_{\Epi(A_N,A_m)}(f,s)\leq 1$, $f \in S_j^1$. So by definition of $\chi$, we have that $\chi(f) \neq k+2$. That is, $i \neq k+2$. But, now $\chi(f)=i$ means $f \in S_i^1$. So $f \in S_i^1 \cap S_j^1$.

    So there is $i \neq j$ such that $S_i^1 \cap S_j^1 \neq \emptyset$ and now by Claim \ref{claim_4_1}, we are done.
\end{proof}

\begin{lem}\label{lem_ardeg_to_synd}
    Suppose that $\limF=\varprojlim(A_n,\varphi_n^{n+1})$ and that $A_m$ has Approximate Ramsey degree $\geq k$ in $\mathcal{F}$. Then, for any $\varepsilon >0$ satisfying the condition below:
    \[\textrm{for all } \mbf{x},\mbf{y} \in \limF, \ \neg (\pi_m(\mbf{x})R^{A_m}\pi_m(\mbf{x})) \implies d_F(q(\mbf{x}),q(\mbf{y}))\geq \varepsilon \]
    we have that there exists $S_1,\ldots, S_k$ syndetic subsets of $\Aut(\limF)$ such that for any $i\neq j$, $s_i \in S_i$, and $s_j \in S_j$, we have $d_{\sup}(\Phi(s_i),\Phi(s_j))\geq \varepsilon$.
\end{lem}

\begin{proof}
    Let $B \in \mathcal{F}$ witness the fact that $A_m$ has Approximate Ramsey degree $\geq k$. That is, for every $C\in \mathcal{F}$ with $\Epi(C,B)\neq \emptyset$ we have a coloring $\chi: \Epi(C,A_m)\to [k]$ such that for any $f \in \Epi(C,B)$ and any $I\subset [k]$ with $\norm{I}=k-1$, $\Epi(B,A_m) \circ f\not\subseteq [\chi^{-1}(I)]_1$. There is no harm in replacing $B$ by any $B' \in \mathcal{F}$ such that $\Epi(B',B) \neq \emptyset$. So we may assume that $B=A_n$ for some $n\geq m$. 

    For any $C \in \mathcal{F}$ say that coloring $\chi:\Epi(C,A_m)\to [k]$ is \emph{anti-Ramsey} if for any $f \in \Epi(C,A_n)$ and $I\subset [k]$ with $\norm{I}=k-1$, $\Epi(A_n,A_m)\circ f\not\subseteq [\chi^{-1}(I)]_1$. 

    Let $l_1,l_2$ be such that $n\leq l_1<l_2$. We say that a coloring $\chi_1:\Epi(A_{l_1}, A_m) \to [k]$ is the \emph{restriction} of a coloring $\chi_2:\Epi(A_{l_2},A_m) \to [k]$ if for any $g \in \Epi(A_{l_1},A_m)$ we have that $\chi_1(g)= \chi_2(g \circ \varphi_{l_1}^{l_2})$. We write $\chi_1 \preceq \chi_2$ if $\chi_1$ is the restriction of $\chi_2$. The claim below follows directly from the definitions and so we omit the proof.

    \begin{claim}\label{claim_lem_1_1}
        Let $n \leq l_1<l_2$. If $\chi_2:\Epi(A_{l_2},A_m) \to [k]$ is anti-Ramsey and $\chi_1:\Epi(A_{l_1},A_m) \to [k]$ is such that $\chi_1 \preceq \chi_2$, then $\chi_1$ is anti-Ramsey.
    \end{claim}

Let $C=\{\chi \ : \ l \geq n, \chi:\Epi(A_l,A_m) \to [k] \textrm{ is anti-Ramsey}\}$. By Claim \ref{claim_lem_1_1} and the fact that there exists an anti-Ramsey coloring $\Epi(A_l,A_m) \to [k]$ for all $l \geq n$, we have that $(C,\preceq)$ is a tree of infinite height. It is finitely branching because for any $l$ the set of function $\Epi(A_l,A_m) \to [k]$ is finite. By Koenig's lemma, let $chi_l:\Epi(A_l,A_m) \to [k]$ for $l\geq n$ such that:
\begin{enumerate}[(i)]
    \item each $\chi_l$ is anti-Ramsey
    \item for $l_1 <l_2$, $\chi_{l_1}\preceq \chi_{l_2}$
\end{enumerate}

We define a coloring $\Delta:\Epi(\limF,A_m) \to [k]$ as follows. For any $f \in \Epi(\limF,A_m)$, there is $l$ (which we can assume is at least $n$) and $f' \in \Epi(A_l,A_m)$ so that $f=f' \circ \pi_l$. Set $\Delta(f)=\chi_l(f')$. This coloring is well-defined by condition (ii) above. 

For $i \in [k]$, define $S_i \subset \Aut(\limF)$ as follows:
\[S_i=\{f  \ : \ \forall h \in \Epi(\limF,A_m) \textrm{ such that }d_{\Epi(\limF,A_m)}(h,\pi_m\circ f)\leq 1, \Delta(h)=i\}\]
Note that if $f \in S_i$, $\Delta(\pi_m\circ f)=i$.

\begin{claim}\label{claim_lem_1_2}
    Each $S_i$ is syndetic.
\end{claim}

\begin{proof}[Proof of Claim \ref{claim_lem_1_2}]
     By Lemma \ref{lem_ext_to_aut}, fix $K \subseteq \Aut(\limF)$ finite such that for every $f \in \Epi(A_n,A_m)$ there is $k \in K$ so that $\pi_m \circ k=f \circ \pi_n$. We claim that $K^{-1}S_i=\Aut(\limF)$. 
         
         Let $g \in \Aut(\limF)$. We will find $k \in K$ such that $kg \in S_i$. Let $p\geq n$ and $g_n^{p}:A_{p} \to A_n$ such that $\pi_n \circ g= g_n^{p} \circ \pi_{p}$. 
         
         By condition (i) above, for each $l \geq p$, $\Epi(A_n,A_m) \circ g_n^p\circ \varphi_{p}^l \not\subseteq [\chi^{-1}([k]\setminus\{i\})]_1$. Thus, for each $l\geq p$, there exists $f_l \in \Epi(A_n,A_m)$ such that for all $h \in \Epi(A_l,A_m)$ with $d_{\Epi(A_l,A_m)}(h,f_l \circ g_n^p \circ \varphi_p^l) \leq 1$, $\chi_l(h)=i$. Since $\Epi(A_n,A_m)$ is finite, there exists $f \in \Epi(A_n,A_m)$ such that for any $l \geq p$ there exists $r \geq l$ such that 
         \begin{equation}\label{eqn_2_1}
         \textrm{for all } h \in \Epi(A_r,A_m) \textrm{ with } d_{\Epi(A_r,A_m)}(h, f \circ g_n^p \circ \varphi_p^r)\leq 1, \chi_r(h)=i
         \end{equation} 
         In fact we claim that for any $l \geq p$, for all $h \in \Epi(A_l,A_m)$ with $d_{\Epi(A_l,A_m)}(h, f\circ g_n^p \circ \varphi_p^l)\leq 1$, $\chi_l(h)=i$. To see this, fix $l \geq p$ and take $r\geq l$ so that Formula \ref{eqn_2_1} holds. Suppose let $h \in \Epi(A_l,A_m)$ with $d_{\Epi(A_l,A_m)}(h, f\circ g_n^p \circ \varphi_p^l)\leq 1$. Then by Remark \ref{rmk_dist_cohere}, we have $d_{\Epi(A_r,A_m)}(h \circ \varphi_l^r, f\circ g_n^p \circ \varphi_p^l \circ \varphi_l^r) \leq 1$. So by Formula \ref{eqn_2_1}, $\chi_r(h \circ \varphi_l^r)=i$. Since by condition (ii), $\chi_l \preceq \chi_r$, we have that $\chi_l(h) = \chi_r(h \circ \varphi_l^r)=i$.  

         From the last two paragraphs, we have $f \in \Epi(A_n,A_m)$ such that for any $l \geq p$, 
         \begin{equation}\label{eqn_2_2}
             \textrm{for all } h \in \Epi(A_l,A_m) \textrm{ with } d_{\Epi(A_l,A_m)}(h, f \circ g_n^p \circ \varphi_p^l)\leq 1, \chi_l(h)=i
         \end{equation}
         Let $k \in K$ such that $\pi_m \circ k= f \circ \pi_n$. We claim $kg \in S_i$. 

         Note that $\pi_m \circ k \circ g=f \circ \pi_n \circ g=f \circ g_n^p \circ \pi_p$. Further for any $l \geq p$, $\pi_m \circ kg= f \circ g_n^p \circ \varphi_p^l \circ \pi_l$. Let $h \in \Epi(\limF,A_m)$ such that $d_{\Epi(\limF,A_m)}(h,\pi_m \circ kg)\leq 1$. Then by definition of $d_{\Epi(\limF,A_m)}$ we have some $l$ which we can assume  is at least $p$ and $h' \in \Epi(A_l,A_m)$ such that $h=h' \circ \pi_l$ and $d_{\Epi(A_l,A_m)}(h', f \circ g_n^p \circ \varphi_p^l) \leq 1$. So by Formula \ref{eqn_2_2}, $\chi_l(h')=i$, which implies that $\Delta(h)=i$. So $kg \in S_i$ as desired.
\end{proof}

\begin{claim}\label{claim_lem_1_3}
    Let $\varepsilon>0$ be such that for all $\mbf{x},\mbf{y} \in \limF$, $\neg(\pi_m(\mbf{x})R^{A_m}\pi_m(\mbf{y})) \implies d_F(q(\mbf{x}),q(\mbf{y}))\geq \varepsilon$.
    Let $i \neq j$. For any $s_i \in S_i$ and $s_j \in S_j$, $d_{\sup}(\Phi(s_i),\Phi(s_j))\geq \varepsilon$.
\end{claim}

\begin{proof}[Proof of Claim \ref{claim_lem_1_3}]
    Let $\varepsilon>0$ be as in the statement of the claim, that is:
    \begin{equation}\label{eqn_lem_1}
        \textrm{for all } \mbf{x},\mbf{y} \in \limF, \neg(\pi_m(\mbf{x})R^{A_m}\pi_m(\mbf{y})) \implies d_F(q(\mbf{x}),q(\mbf{y}))\geq \varepsilon
    \end{equation}
    and let $i\neq j$. Let $s_i \in S_i$ and $s_j \in S_j$.
    Note that since $\Delta(s_i)=i$ and $\Delta(s_j)=j$, we have that $d_{\Epi(\limF,A_m)}(\pi_m \circ s_i,\pi_m \circ s_j)>1$. So there is $l \geq m$ and $s_i',s_j' \in \Epi(A_l,A_m)$ with $\pi_m \circ s_i=s_i' \circ \pi_l$ and $\pi_m \circ s_j=s_j' \circ \pi_l$ and $d_{\Epi(A_l,A_m)}(s_i',s_j')>1$. This implies there exists $a \in A_l$ such that $d_{A_m}(s_i'(a),s_j'(a))>1 \implies \neg (s_i'(a) R^{A_m}s_j(a))$. Let $\mbf{x} \in \limF$ such that $\pi_l(\mbf{x})=a$. Then, $\pi_m \circ s_i(\mbf{x})=s_i'(a)$ and $\pi_m \circ s_j(\mbf{x})=s_j'(a)$, so $\neg (\pi_m(s_i(\mbf{x})),\pi_m(s_j(\mbf{x})))$. By Formula \ref{eqn_lem_1} this implies $d_F(q(s_i(\mbf{x})),q(s_j(\mbf{x})))\geq \varepsilon$. By definition of $\Phi$ we get $d_F(\Phi(s_i)(q(\mbf{x})),\Phi(s_j)(q(\mbf{x})))\geq \varepsilon$. Thus $d_{\sup}(\Phi(s_i),\Phi(s_j))\geq \varepsilon$.
\end{proof}
    Claims \ref{claim_lem_1_2} and \ref{claim_lem_1_3} imply that $S_1,\ldots, S_k$ satisfy the lemma.
\end{proof}

\section{Main theorems}\label{sec_main_thms}
\subsection{Extreme amenability}\label{sec_ea}
Here is the main theorem about extreme amenability in the context of projective Fraisse limits.

\begin{thm}\label{thm_main_ea}
Let $\mathcal{F}$ be a transitive projective Fraisse category with limit $\limF$. Let $\Phi:\Aut(\mathbb{F}) \to \Homeo(F)$ be the map induced by the quotient map, where $F=\mathbb{F} / R^{\mathbb{F}}$. Let $\tau$ be the topology on $\Aut(\limF)$ given by $\tau=\{\Phi^{-1}(U) \ : \ U \in \tau_{\Homeo(F)}\}$ where $\tau_{\Homeo(F)}$ is the uniform convergence topology on $\Homeo(F)$ induced by the supremum metric. The category $\F$ has the approximate Ramsey property if and only if $(\Aut(\limF),\tau_{\Homeo})$ is extremely amenable. 
\end{thm}

Before we prove the theorem we point out the corollary of interest from the perspective of studying homeomorphism groups of connected spaces:

\begin{cor}\label{cor_main_ea}
    Let $\mathcal{F},\limF,F$ as in the statement of Theorem \ref{thm_main_ea}. The category $\F$ has the approximate Ramsey property if and only if $\overline{\Phi[\Aut(\limF)]}$ is an extremely amenable subgroup of $\Homeo(F)$. 
\end{cor}

\begin{proof}[Proof of Corollary \ref{cor_main_ea}]
    This follows from Theorem \ref{thm_main_ea} and Proposition \ref{prop_dense_subgp}.
\end{proof}

\begin{proof}[Proof of $\implies$ of Theorem \ref{thm_main_ea}]
    Assume that $\mathcal{F}$ has the approximate Ramsey property. We will show that $(\Aut(\limF),\tau)$ is extremely amenable by checking it has the condition from Theorem \ref{thm_bart_1}. Let $d_{\sup}$ be the supremum metric on $\Homeo(F)$. 

    Let $V\subseteq \Aut(\limF)$ be an open neighborhood of the identity in $(\Aut(\limF),\tau)$ and let $S,T$ be open syndetic subsets of $\Aut(\limF)$. We want to show that $VS \cap VT \neq \emptyset$. We may assume that there is $\varepsilon>0$ such that:
    \[V=\Phi^{-1}(B_{d_{\sup}}(1,\varepsilon))=\{f \in \Aut(\limF) \ : \ d_{\sup}(\Phi(f),1)<\varepsilon\}\]
    We want to show that $VS \cap VT \neq \emptyset$. By Remark \ref{rmk_vs} it suffices to show that there is $s \in S$ and $t \in T$ such that $d_{\sup}(\Phi(s),\Phi(t))<\varepsilon$.

     Following the notation from Section \ref{subsec_ntn_limit}, let $\limF = \varprojlim (A_n,\varphi_n^{n+1})$ with $A_n$ and $\varphi_n^{n+1}$ in $\mathcal{F}$. Let $\pi_n:\limF \to A_n$ be the $n$th projection map and $q:\limF \to F = \limF / R^{\limF}$ the quotient map. By Lemma \ref{lem_graphdistsmall}, let $m \in \N$ such that for all $\bf{x},\bf{y} \in \limF$:
    \begin{equation}\label{eqn_thm_1_1} 
    d_{A_{m}}(\pi_{m}(\mathbf{x}),\pi_{m}(\mathbf{y}))\leq 2\implies d_F(q(\mbf{x}),q(\mbf{y}))<\varepsilon 
    \end{equation}
    Since $\mathcal{F}$ has the approximate Ramsey property, $\rdeg(A_m)=1$ in $\mathcal{F}$. So by Equation \ref{eqn_thm_1_1} and Lemma \ref{lem_ardeg_synd_overlap}, there is $s \in S$ and $t \in T$ with $d_{\sup}(\Phi(s),\Phi(t))<\varepsilon$ and this finishes the proof.     
\end{proof}

 \begin{proof}[Proof of $\impliedby$ of Theorem \ref{thm_main_ea}]
     We will prove the contrapositive. Assume that $\mathcal{F}$ fails to have the Approximate Ramsey property. Then, let $A \in \mathcal{F}$ with Approximate Ramsey degree of $A \geq 2$ in $\mathcal{F}$. We may assume $A=A_m$ where $\limF=\varprojlim(A_n,\varphi_n^{n+1})$. By Lemma \ref{lem_metric_spread}, let $\varepsilon>0$ such that for all $\mathbf{x},\mathbf{y} \in \limF$, $\neg (\pi_m(\mbf{x})R^{A_m}\pi_m(\mbf{y})) \implies d_F(r(\mbf{x}),q(\mbf{y}))\geq \varepsilon$. By Lemma \ref{lem_ardeg_to_synd}, there exists $S_1,S_2$ syndetic subsets of $\Aut(\limF)$ such that for all $s_1 \in S_1, s_2 \in S_2$, $d_{\sup}(\Phi(s_1),\Phi(s_2)) \geq \varepsilon$.

     Now, let $V=\Phi^{-1}(B_{d_{\sup}}(1,\varepsilon/4))$. Then, $V$ is open in $(\Aut(\limF),\tau)$ and so is $VS_1$ and $VS_2$. Each $VS_i \supseteq S_i$ and so is syndetic.  By Remark \ref{rmk_vs}, $V(VS)S \cap V(VT)=\emptyset$. So by Theorem \ref{thm_bart_1}, $(\Aut(\limF),\tau)$ is not extremely amenable.
 \end{proof}  

 \subsection{Metrizability of the universal minimal flow}\label{sec_umf}
Now we prove the appropriate generalization of Theorem \ref{thm_main_ea} but for metrizability of the universal minimal flow.

\begin{thm}\label{thm_main_umf}
Let $\mathcal{F}$ be a transitive projective Fraisse category with limit $\limF$. Let $\Phi:\Aut(\mathbb{F}) \to \Homeo(F)$ be the map induced by the quotient map, where $F=\mathbb{F} / R^{\mathbb{F}}$. Let $\tau$ be the topology on $\Aut(\limF)$ given by $\tau=\{\Phi^{-1}(U) \ : \ U \in \tau_{\Homeo(F)}\}$ where $\tau_{\Homeo(F)}$ is the uniform convergence topology on $\Homeo(F)$ induced by the supremum metric. Every $A \in \mathcal{F}$ has finite approximate Ramsey degree if and only if $(\Aut(\limF),\tau_{\Homeo})$ has metrizable universal minimal flow. 
\end{thm}

In the proof below we use for any $k \in \N$, $[k]=\{1,2,\ldots,k\}$.

\begin{proof}[Proof of $\implies$ of Theorem \ref{thm_main_umf}]
    Assume that every $A \in\mathcal{F}$ has finite Approximate Ramsey degree. We will show that $(\Aut(\limF),\tau)$ has metrizable universal minimal flow by checking the condition from Theorem \ref{thm_dis}. Let $d_{\sup}$ as usual be the supremum metric on $\Homeo(F)$.

    Let $V \subseteq (\Aut(\limF),\tau)$ be an open neighborhood of the identity. We may assume that $V=\Phi^{-1}(B_{d_{\sup}}(1,\varepsilon))$ for some $\varepsilon >0$. 

    Let $\limF=\varprojlim(A_n,\varphi_n^{n+1})$. By Lemma \ref{lem_graphdistsmall}, let $m \in \N$ so that for all $\mbf{x},\mbf{y} \in \limF$:
    \begin{equation}\label{eqn_3_1}
        d_{A_{m}}(\pi_{m}(\mbf{x}),\pi_{m}(\mbf{y})) \leq 2 \implies d_F(q(\mbf{x}),q(\mbf{y}))<\varepsilon
    \end{equation}

    Let $k=\rdeg(A_{m})$. We claim for any $S_,\ldots, S_{k+1}$ open syndetic in $(\Aut(\limF),\tau)$, there is $i \neq j$ so that $VS_i \cap VS_j \neq \emptyset$. Once we prove this, we are done by Theorem \ref{thm_dis}.

    So let $S_1,\ldots, S_{k+1}$ be open and syndetic in $(\Aut(\limF),\tau)$. We want to show there is $i \neq j$ so that $VS_i \cap VS_j \neq \emptyset$. By Remark \ref{rmk_vs}, it suffices to show that there is $i \neq j$ and $s_i \in S_i$ and $s_j \in S_j$ with $d_{\sup}(\Phi(s_i),\Phi(s_j))<\varepsilon$. This follows by the fact that $\rdeg(A_m) \leq k$, Equation \ref{eqn_3_1} and Lemma \ref{lem_ardeg_synd_overlap}.
\end{proof}

\begin{proof}[Proof of $\impliedby$ of Theorem \ref{thm_main_umf}]
    We prove the contrapositive. Let $A \in \mathcal{F}$ have infinite Approximate Ramsey degree in $\mathcal{F}$. Any $B \in \mathcal{F}$ with $\Epi(B,A) \neq \emptyset$ will also have infinite Ramsey degree in $\mathcal{F}$, so we may assume that $A=A_m$ where $\limF=\varprojlim(A_n,\varphi_n^{n+1})$. By Lemma \ref{lem_metric_spread}, fix $\varepsilon >0$ such that for all $\mbf{x},\mbf{y} \in \limF$, 
    \[\ \neg (\pi_m(\mbf{x})R^{A_m}\pi_m(\mbf{x})) \implies d_F(q(\mbf{x}),q(\mbf{y}))\geq \varepsilon\]

    Let $V=\Phi^{-1}(B_{d_{\sup}}(1,\varepsilon/4))$. Let $k \in \N$. Since $\rdeg(A_m) \geq k$, by Lemma \ref{lem_ardeg_to_synd} there exists $S_1,\ldots, S_k$ syndetic subsets of $\Aut(\limF)$ such that for all $i \neq j$, $s_i \in S_i, s_j\in S_j$, we have $d_{\sup}(\Phi(s_i),\Phi(s_j))\geq \varepsilon$. So $VS_1,\ldots,VS_k$ are open syndetic subsets of $(\Aut(\limF),\tau)$ such that by Remark \ref{rmk_vs} $V(VS_1), \ldots, V(VS_k)$ are pairwise disjoint. Since this holds for any $k \in \N$, by Theorem \ref{thm_dis} $(\Aut(\limF),\tau)$ has non-metrizable universal minimal flow.
\end{proof}

\begin{cor}\label{cor_main_metr_umf}
    Let $\mathcal{F},\limF,F$ as in the statement of Theorem \ref{thm_main_umf}. Every $A \in \mathcal{F}$ has finite approximate Ramsey degree in $\mathcal{F}$ if and only if $\overline{\Phi[\Aut(\limF)]}$ with the topology inherited as a closed subgroup of $\Homeo(F)$ has metrizable universal minimal flow.
\end{cor}

\section{Application to universal Pseudo-solenoid}\label{sec:applications}

A \textbf{continuum} is a compact, connected, metrizable space. If $X$ is a continuum, a \textbf{subcontinuum} of $X$ is $Y\subseteq X$ where $Y$ is compact, connected. We say a continuum $X$ is \textbf{indecomposable} if whenever $X=A \cup B$ and $A,B$ are subcontinua of $X$ we have that $A=X$ or $B=X$. A continuum is \textbf{hereditarily indecomposable} if every subcontinuum of it is indecomposable.

A \textbf{chain} is a collection of open sets $U_1,\ldots, U_n$ such that $U_i \cap U_j \neq \emptyset \iff \norm{i-j}\leq 1$. A \textbf{circular chain} is a collection of open sets $U_1,\ldots, U_n$ such that $U_i \cap U_j \neq \emptyset \iff (\norm{i-j}\leq 1 \textrm{ or } \{i,j\}=\{1,n\})$. A continuum $X$ is \textbf{chainable} (resp. \textbf{circularly chainable}) if every open cover of $X$ is refined by a chain (resp. circular chain). 

\begin{defn}\label{defn_solenoid}
    A \textbf{pseudo-solenoid} is a continuum which is circularly chainable, not chainable, and hereditarily indecomposable.
\end{defn}

\begin{defn}(Rogers \cite{rogers})
    The \textbf{universal pseudo-solenoid} is the unique up to homeomorphism pseudo-solenoid which continuously surjects onto any circularly chainable, not chainable continuum. 
\end{defn}

Note that a universal pseudo-solenoid continuously surjects onto any pseudo-solenoid. For more on the classification of pseudo-solenoids up to homeomorphism and the universal pseudo-solenoid, see \cite{fearnley,rogers}. In the literature some authors use the term ``pseudo-circle'' and some use the term pseudo-solenoid. For the remainder of this section, let $S$ be the universal pseudo-solenoid. 

Irwin in \cite{irwin_thesis} constructed a projective Fraisse category which approximates $S$ and we review his construction now. 

A finite graph $(A,R^A)$ is a \textbf{cycle} if for every vertex $a \in A$, there are exactly two vertices other than $a$ which are $R^A$-related to $a$. We will typically think of a cycle $A$ with $n$ vertices as with a labeling of its vertices by the numbers $0,1,\ldots,n-1$ such that $i R^Aj \iff \norm{i-j}\leq 1 \textrm{ or } \{i,j\}=\{0,n-1\}$.

Let $(C,R^C), (D,R^D)$ be cycles and let $f:D \to C$ be a surjective morphism. Let $C$ have $m$ vertices labeled by $0,1,\ldots, m-1$ and $D$ have $n$ vertices labeled by $0,1,\ldots, n-1$ as described above. 

Let $i,j \in \{0,1,\ldots,n-1\}$ with $iR^D j$. We say that $(i,j)$ is \textbf{positively oriented for $f$} iff $(f(i),f(j)) \in \{(k,k+1) \ : \ k=0,1,\ldots, m-2 \}\cup \{(m-1,0)\}$. We say that $(i,j)$ is \textbf{constant for $f$} if $f(i)=f(j)$. Otherwise we say that $(i,j)$ is \textbf{negatively oriented for $f$}.

Let $D^{\textrm{path}}=\{(i,i+1) \ : \ i =0,1,\ldots, n-2\} \cup\{(n-1,0)\}$. The edges in $D^{\textrm{path}}$ collect one full ``path'' around cycle $D$.

Let 
\[D_f^+=\{(i,j) \in D^{\textrm{path}} \ : \ (i,j) \textrm{ is positively oriented for $f$}\}\]
and let
\[D_f^-=\{(i,j) \in D^{\textrm{path}} \ : \ (i,j) \textrm{ is negatively oriented for $f$}\}\]

Define the \textbf{degree of $f$}, denoted $\deg(f)$ to be 
\[\deg(f)= \frac{\norm{\norm{D_f^+}-\norm{D_f^-}}}{m}\]

In words, the degree of $f$ is the number of times that morphism $f$ wraps the cycle $D$ around the cycle $C$. Notice that for any surjective morphism $f$, $\deg(f)$ is a natural number. 


A useful fact is that degree is multiplicative: 

\begin{lem}[\cite{irwin_thesis} Lemma 4.3]\label{lem_degree_multiplicative}
    Let $f:D \to C$, $g:C \to B$ be  surjective morphisms between cycles. Then $\deg(g \circ f) =\deg (g) \deg(f)$.
\end{lem}

\begin{defn}[Irwin \cite{irwin_thesis}]
    Let $\mathcal{S}$ be the category whose objects are cycles and morphisms are  surjective morphisms of non-zero degree. 
\end{defn}

By \cite{irwin_thesis} Theorem 4.9, $\mathcal{S}$ is a transitive projective Fraisse category. Let $\mathbb{S}$ be the limit of $\mathcal{S}$. Then by \cite{irwin_thesis} Theorem 4.20, $\mathbb{S} / R^{\mathbb{S}}$ is the universal pseudo-solenoid. Let $q:\mathbb{S} \to \mathbb{S} / R^{\mathbb{S}}=S$ be the quotient map and  let $\Phi:\Aut(\limS) \to \Homeo(S)$ be the map induced by the quotient map. The following lemma is derived from a combination of results from \cite{irwin_thesis}:

\begin{lem}[Irwin \cite{irwin_thesis}]\label{lem_fraisse_dense_pseudo_sol}
    The subgroup $\Phi[\Aut(\limS)]$ is a dense subset of $\Homeo(S)$.
\end{lem}

\begin{proof}
    An inspection of the proof of Theorem 4.22 of \cite{irwin_thesis} shows that the $h^*$ in the statement of Theorem 4.22 is in $\Phi[\Aut(\limS)]$. So by Theorem 4.22 of \cite{irwin_thesis} it suffices to show that any homeomorphism $f:S \to S$ is of positive rank (see Definition 4.8 of \cite{irwin_thesis} for the definition of positive rank). By \cite{irwin_thesis} Theorem 4.16 and Lemma 4.17, let $\delta >0$ be such that every circular chain on $S$ whose elements all have diameter less than $\delta$ is a positive chain. Let $\mathcal{U}$ be any open cover of $S$. Since $f^{-1}$ is uniformly continuous, let $\varepsilon >0$ so that $d_S(x,y)<\varepsilon \implies d_S(f^{-1}(x),f^{-1}(y))<\delta$. Take $\mathcal{C}$ to be a circular chain refining $\mathcal{U}$ and also having the property that each element of $\mathcal{C}$ has diameter less than $\epsilon$. Then $\{f^{-1}(C)\ : \ C \in \mathcal{C}\}$ is a chain with every member having diameter less than $\delta$ and hence is a positive chain. This shows that $f$ has positive rank.
\end{proof}

The anti-Ramsey result we will use relies on the lemma below which is essentially a combinatorial version of the fact that degree is a continuous invariant on the space of continuous functions from the circle to itself with the uniform convergence topology:

\begin{lem}\label{lem_deg_cont}
    Let $(C,R^C)$ be the cycle with 5 vertices. If $f,g :(D,R^D) \to (C,R^C)$ are morphisms in $\mathcal{S}$ such that $d_{\Epi(D,C)}(f,g) \leq 1$, then $\deg(f)=\deg(g)$.
\end{lem}

\begin{proof}
    Let $f,g$ be in $\Epi(D,C)$.  We will assume at first that there is $d \in D$ with $f(d)=g(d)$. Then, at the end we address why this assumption is okay to make. We label the vertices of $D$ by $0,1,2,\ldots, n-1$ in the usual way where $iR^Dj \iff \norm{i-j}\leq 1 \textrm{ or }\{i,j\}=\{0,n-1\}$ and so that $d$ corresponds to $0$. Similarly we assume $C$ has vertices labeled by $0,1,2,3,4$. So $f,g$ are maps $\{0,1,2,\ldots, n-1\} \to \{0,1,2,3,4\}$ with $f(0)=g(0)$. 

    We now define maps $\tilde{f}$, $\tilde{g}$ which go from $\{0,1,\ldots,n\}$ to $\Z$. The maps $\tilde{f}$,$\tilde{g}$ can be thought of as ``lifts'' of $f,g$ in the sense of classical algebraic topology of circle maps. Here is the definition which is given inductively on $i$:
    Define $\tilde{f}(0)=f(0)$. 
    Given that $\tilde{f}(i)$ is defined where $i<n-1$, we define: 
    \[\tilde{f}(i+1)= 
    \begin{cases}
        \tilde{f}(i)+1 & \textrm{ if } (i,i+1) \textrm{ is positively oriented for }f\\
        \tilde{f}(i) & \textrm{ if } (i,i+1) \textrm{ is constant for }f\\
        \tilde{f}(i)-1 & \textrm{ if } (i,i+1) \textrm{ is negatively oriented for }f\\
    \end{cases}
    \]
    We define 
    \[\tilde{f}(n)= 
    \begin{cases}
        \tilde{f}(n-1)+1 & \textrm{ if } (n-1,0) \textrm{ is positively oriented for }f\\
        \tilde{f}(n-1) & \textrm{ if } (n-1,0) \textrm{ is constant for }f\\
        \tilde{f}(n-1)-1 & \textrm{ if } (n-1,0) \textrm{ is negatively oriented for }f\\
    \end{cases}
    \]
    
    We let $\tilde{g}:\{0,1,\ldots, n\} \to \Z$ be defined analogously for $g$.

    Notice that $\deg(f)= \frac{\norm{\tilde{f}(n)-\tilde{f}(0)}}{5}$ and similarly $\deg(g)= \frac{\norm{\tilde{g}(n)-\tilde{g}(0)}}{5}$. 
    
    For any $i \in \{0,1,\ldots, n-1\}$, 
    \[\tilde{f}(i) \pmod 5=f(i)\]
    and
    \[\tilde{g}(i) \pmod 5 =g(i)\]
    and 
     \[\tilde{f}(n) \pmod 5 =f(0), \ \tilde{g}(n) \pmod 5 =g(0) \]

    \begin{claim}
        We have that for any $i \in \{0,1,2\ldots, n\}$, $\norm{\tilde{f}(i)-\tilde{g}(i)} \leq 1$. 
    \end{claim}

    \begin{proof}
        We prove this by induction on $i$. First note that $\tilde{f}(0)=f(0)=g(0)=\tilde{g}(0)$. Suppose the claim holds for some $i$. This means that $\norm{\tilde{f}(i)-\tilde{g}(i)} \leq 1$. Since $d_{\Epi(D,C)}(f,g)\leq 1$,  $f(i+1) R^C g(i+1)$, and so we have that 
        
        \begin{equation}\label{eqn_new_1}
        \begin{split}
            \norm{\tilde{f}(i+1) \pmod 5-\tilde{g}(i+1) \pmod 5} \leq 1 \\ \textrm{ or }  \{\tilde{f}(i+1) \pmod 5,\tilde{g}(i+1) \pmod 5\}=\{0,4\}
        \end{split}
        \end{equation}
        
        Now the fact that and $\norm{\tilde{f}(i)-\tilde{g}(i)} \leq 1$, Equation \ref{eqn_new_1}, and $\norm{\tilde{f}(i+1)-\tilde{f}(i)} \leq 1$ (by definition of $\tilde{f}$) and  $\norm{\tilde{g}(i+1)-\tilde{g}(i)} \leq 1$ imply that $\norm{\tilde{f}(i+1)-\tilde{g}(i+1)} \leq 1$.
    \end{proof}

    Since $\norm{\tilde{f}(n)-\tilde{g}(n)} \leq 1$, we have that $\norm{\deg(f)-\deg(g)} \leq \frac{1}{5}$ and so since $\deg$ takes values in $\N$, we have that $\deg(f)=\deg(g)$.

    In general, let $f,g:D \to C$ with $d_{\Epi(D,C)}(f,g)\leq 1$. Then, consider the map $\varphi: (E,R^E) \to D$ where $E$ is a cycle with $3n$ vertices defined as follows. Let $E=\{0,1,2,\ldots, 3n-1\}$ and consider $E$ as divided into intervals of the form $I_j=\{3j,3j+1,3j+2\}$ as $j$ ranges over $0,1,\ldots n-1$. and define 
    \[\varphi(i)= j \iff i \in I_j  \]
    It is not hard to see that $\varphi$ is in $\mathcal{S}$ and that $\deg(\varphi)=1$. Let $f_1=f \circ \varphi$ and $g_1=g \circ \varphi$. Note that $d_{\Epi(E,C)}(f_1,g_1)\leq 1$ and by Lemma \ref{lem_degree_multiplicative}, $\deg(f_1)=\deg(f)$ and $\deg(g_1)=\deg(g)$. Now consider $f_2$ defined by $f_2=f_1$ on $E\setminus \{1\}$ and $f_2(1)=g_1(1)$. Note $f_1(0)=f_1(1)=f_1(2)=k$ for some $k$ and $g_1(1)=j$ where $jR^Ck$ since $d_{\Epi(E,C)}(f_1,g_1) \leq 1$. So $f_2(0)R^Cf_2(1)R^Cf_2(2)$. Since 1 was the only value of the domain changed in $f_2$, this shows $f_2$ is still an epimorphism. Further $\deg(f_2)=\deg(f_1)=\deg(f)$ since the construction of $f_2$ replaced one constant edge of $f_1$ by a positively oriented edge and one constant edge of $f_1$ by a negatively oriented edge. We still have $d_{\Epi(f_2,g_1)} \leq 1$ since $f_2(1)=g_1(1)$ and that was the only change made from $f_1$. So we have $f_2$ and $g_1$ satisfy the assumption we made in the beginning: they are in $\Epi(E,C)$, $d_{\Epi(E,C)}(f_2,g_1)\leq 1$, there is $x \in E$ with $f_2(x)=g_1(x)$. Also, $\deg(f_2)=\deg(f)$ and $\deg(g_1)=\deg(g)$. This finishes the proof of the lemma.
\end{proof}

Now we can prove the main application:

\begin{thm}\label{thm_pseudo_sol}
    Let $S$ be the universal pseudo-solenoid. The group $\Homeo(S)$ of homeomorphisms of $S$ has non-metrizable universal minimal flow.
\end{thm}

\begin{proof}
    By Corollary \ref{cor_main_metr_umf}, Proposition \ref{prop_dense_subgp}, and Lemma \ref{lem_fraisse_dense_pseudo_sol} it suffices to show that the projective Fraisse category $\mathcal{S}$ has an object of infinite Approximate Ramsey degree. We claim that the cycle $C$ with five vertices has infinite Approximate Ramsey degree.

    Let $k \in \N$. We will show that the cycle $D$ with $5\cdot 2^{k}$ vertices witnesses that  $\rdeg(C) \geq k$. Notice that for $i=1,\ldots,2^k$ there exists a morphism $\phi \in \Epi(D,C)$ with $\deg(\varphi)=i$.

    Let $E$ be any cycle with $\Epi(E,D) \neq \emptyset$. We color $\Epi(E,C)$ as follows. For any $n\in\N$, let $\rho(n)$ be the maximum natural number such that $2^{\rho(n)} \vert n$. That is, $\rho(n)$ is the number of 2's that occur in the prime decomposition of $n$. Color by $\chi:\Epi(E,C) \to \{0,1,2,\ldots, k-1\}$ defined by
    \[\chi(f)=\rho(\deg(f))\pmod k\]
    Let $g \in \Epi(E,D)$. Let $i \in \{0,1,\ldots,k-1\}$. Then, let $\varphi \in \Epi(D,C)$ with $\deg(\varphi) = 2^j$ where $j=i-\rho(\deg(g)) \pmod k$. Then:
    \begin{align*}
         \chi(\varphi \circ g) &= \rho(\deg(\varphi \circ g)) \pmod k\\
         &= \rho(\deg(g) \deg(\varphi)) \pmod k \\
         &=\rho(\deg(g))+\rho(\deg(\varphi)) \pmod k\\
         &=i\\
    \end{align*}
    Further, for any $h \in \Epi(E,C)$ with $d_{\Epi(E,C)}(\varphi \circ g,h) \leq 1$ by Lemma \ref{lem_deg_cont} has $\deg(h)=\deg(\varphi\circ g)$ and thus $\chi(h)=i$. So, 
    \[\Epi(D,C) \circ g \not \subseteq[\chi^{-1}(\{0,1,\ldots,k-1\}\setminus \{i\})]_1\]
    Since this holds for all $i$, we have that for any $I \subset \{0,1,\ldots,k-1\}$ with $\norm{I}\leq k-1$, $\Epi(D,C) \circ g\not \subseteq [\chi^{-1}(I)]_1$ and so $\rdeg(C) \geq k$. 
\end{proof}

\printbibliography

\end{document}